\documentclass[11pt, reqno]{amsart}

\usepackage{amssymb, amsfonts, euscript, enumerate}

\usepackage{pdfsync}

\newcommand{\rem}[1]{\relax}
\usepackage{ifthen, xspace}
    \newcommand{\bsc}{\usefont{T1}{cmr}{bx}{sc}}
    \newcommand{\pending}[1][]{{\noindent\highlight{\bsc\ifthenelse{\equal{#1}{}}{[to be written]}{[to be written: {\rm #1}]}}}\xspace}
    \newcommand{\prelims}[1][]{{\noindent\highlight{\bsc\ifthenelse{\equal{#1}{}}{[add to prelims]}{[add to prelims: {\rm #1}]}}}\xspace}
\usepackage{color}
    \newcommand{\highlight}[2][red]{{\color{#1}#2}}

	\usepackage{xcolor}
	\usepackage{soul}
	\definecolor{lightred}{rgb}{1, .60,.60}
	
\renewcommand{\hat}{\widehat}

\newcommand{\ra}{\rangle}
\newcommand{\la}{\langle}

\newcommand{\msf}{\mathsf}
\newcommand{\WKL}{\msf{WKL}}
\newcommand{\WKLO}{\msf{WKL_0}}
\newcommand{\ACA}{\msf{ACA}}
\newcommand{\RCA}{\msf{RCA}}
\newcommand{\RCAO}{\msf{RCA_0}}
\newcommand{\PA}{\mathrm{PA}}

\newcommand{\M}{\mathcal{M}}

\newcommand{\seq}[1]{{\left\langle{#1}\right\rangle}}
\newcommand{\setN}{ \mathbb{N}}
\newcommand{\setR}{ \mathbb{R}}

\newcommand{\until}[1]{\upharpoonright_{#1}}

\newcommand{\zj}{\emptyset'}
\newcommand{\Cs}{2^{\omega}}
\newcommand{\strs}{2^{<\omega}}
\newcommand{\cl}{\mathrm{cl}}

\theoremstyle{plain}
\newtheorem{theorem}{Theorem}
\newtheorem{lemma}[theorem]{Lemma}
\newtheorem{corollary}[theorem]{Corollary}
\newtheorem{proposition}[theorem]{Proposition}

\theoremstyle{definition}
\newtheorem{definition}[theorem]{Definition}
\newtheorem{question}{Question}

\theoremstyle{remark}
\newtheorem*{claim}{Claim}

\numberwithin{equation}{section}

\begin{document}

\title{On The Strength of Two Recurrence Theorems}

\renewcommand{\datename}{Last compilation:}

\author{Adam R. Day}
\address{Department of Mathematics\\
University of California, Berkeley\\
Berkeley, CA 94720-3840 USA}
\email{adam.day@math.berkeley.edu}
\thanks{The author was supported by a Miller Research Fellowship in the Department of Mathematics at the University of California, Berkeley. The author thanks Ted Slaman for numerous discussions on this material and Denis Hirschfeldt for introducing him to reverse mathematics at the 2010 Asian Initiative for Infinity Graduate Summer School, organized by the Institute for Mathematical Sciences in Singapore.}



\maketitle

\begin{abstract}
This paper uses the framework of reverse mathematics to  investigate the strength of two recurrence theorems of topological dynamics. It establishes that one of these theorems, the existence of an almost periodic point, lies strictly between $\WKL$ and $\ACA$ (working over $\RCA_0$). This is the first example of a theorem with this property.  It also shows the existence of an almost periodic point is conservative over $\RCA_0$ for $\Pi^1_1$ sentences. These results establish  the existence of a new upwards-closed subclass of the $\PA$ degrees.
\end{abstract}

\section{Introduction}
\noindent Dynamical systems are studied by different branches of mathematics in many different forms. 
In the simplest setting, a dynamical system $(X, T)$ is comprised of a set $X$ and a transformation $T: X\rightarrow X$. By placing different requirements on $X$ and $T$,  structure can be added to the system that will influence its behavior.  

Central to the analysis of a dynamical system is the analysis of the orbits of points in the system. Given a point $x \in X$, the orbit of $x$ is the sequence $x, T(x), T^2(x), \ldots.$ If our dynamical system has certain global properties, then this  guarantees the existence of points with certain orbits.

\begin{theorem}[Birkhoff's recurrence theorem]
Let $X$ be a compact topological space and $T:X\rightarrow X$ a continuous transformation. Then there exists $x \in X$ and a sequence $n_1, n_2, \ldots$, such that 
\[\lim_i T^{n_i}(x)\rightarrow x.\]
\end{theorem}
Such an $x$ is called a \textit{recurrent point} of the system $(X,T)$. Comparable results hold  if we place a probability measure on the space $X$ and require $T$ be a measure-preserving transformation.  

The standard proof of Birkhoff's recurrence theorem, shows the existence of a point $x$ with the following stronger property (see for example \cite[Theorem 2.3.4]{Tao_2009}). For every neighborhood $N$ of $x$, there is a bound $b$, such that for all $n$, there is a $k <b$ with $T^{n+k}(X) \in N$. Such a point $x$ is called an \textit{almost periodic point} of the system $(X,T)$.

The objective of this paper is to analyze the reverse mathematical strength of the existence of recurrent points and almost periodic points. The motivation for this work  lies not just in the intrinsic interest of Birkhoff's recurrence theorem but in the fact that this is the simplest of a family of recurrence theorems, that have widespread applications. 
In this respect,  Birkhoff's recurrence theorem is similar to Ramsey's theorem and the reverse mathematical study of Ramsey's theorem has been remarkable fruitful.   
Two examples will illustrate the importance of recurrence theorems. Firstly,  Furstenberg's multiple recurrence theorem, a theorem of measure-preserving systems  can be used to prove Szemeredi's theorem.
A second example, which has been studied from a reverse mathematical perspective, is the Auslander-Ellis theorem. This theorem states that if $X$ is compact metric space, with metric $d$, and $T:X\rightarrow X$ is a continuous transformation, then for any point $x$, there exists a point $y$ such that:
\begin{enumerate}
\item $y$ is an almost periodic point of the system.
\item $(\forall \epsilon)(\exists n)(d(T^n(x), T^n(y)) < \epsilon)$.
\end{enumerate}
Blass, Hirst and Simpson have shown that $\ACA_0^{+}$ proves the Auslander-Ellis theorem \cite{Blass_Hirst_Simpson_1987}. It is an open question as to whether or not it follows from $\ACA_0$ \cite{Montalban_2011}. 
The Auslander-Ellis theorem can be used to prove Hindman's theorem. Hindman's theorem states that if the integers are colored with finitely many colors, then there exists an infinite set $S$ such that $\{n \colon n$ is a finite sum of elements of $S\}$ is homogenous.
 Blass, Hirst and Simpson also showed that the strength of Hindman's Theorem lies between $\ACA_0^{+}$ and $\ACA_0$ \cite{Blass_Hirst_Simpson_1987}. Recent work, particular of Towsner, has shed further light on the difficult question of  calibrating the  strength of Hindman's theorem \cite{Beig_Town_2012, Towsner_2011b,Towsner_2011a,Towsner_2012}.

In this paper we will investigate topological dynamical systems where $X$ is a closed subset of Cantor space and $T$ is a continuous transformation. This is a very important class of topological dynamical systems because subsets of natural numbers can be coded as elements of Cantor space. The proof of Hindman's theorem via  the Auslander-Ellis theorem uses such systems. In the next section we will develop and formalize two principles. 
\begin{enumerate}
\item $\msf{RP}$: Every topological dynamical system on Cantor space contains a recurrent point.
\item $\msf{AP}$: Every topological dynamical system on Cantor space contains an almost periodic point.
\end{enumerate}
In Section~\ref{sect: RP}, we will show that over $\RCA_0$, $\msf{RP}$ is equivalent to $\WKL$. This is perhaps a little surprising because the set of recurrent points of a system is not closed. 
The principal $\msf{AP}$ is more unusual. 
In Section~\ref{section min}, we analyze the standard proof of the existence of an almost periodic point and show this requires $\ACA_0$ to carry out.
However, from the perspective of reverse mathematics,  this proof is not optimal. In fact, over $\RCA_0$, the principle $\msf{AP}$ lies \textit{strictly} between $\WKL$ and $\ACA$. This is the first natural example of a principle with this property. The separation between $\msf{AP}$ and $\WKL$ is established in Section~\ref{sect: AP WKL sep}, and the separation between $\ACA$ and $\msf{AP}$ is established in Section~\ref{sect: ACA AP}. 
  
 Harrington proved that $\WKL_0$ is conservative for $\Pi^1_1$ sentences over $\RCA_0$. In Section~\ref{sect: conservation}, we show that $\RCA_0 +\msf{AP}$ also has this property.
 
 The $\PA$ degrees are those Turing degrees that contain a complete extension of Peano Arithmetic. This is a very well-studied class of upwards-closed Turing degrees. 
 From a computability-theoretic perspective, the proof that the principle $\msf{AP}$ lies strictly between $\WKL$ and $\ACA$  establishes the existence of an interesting  upwards-closed strict subclass of the $\PA$ degrees. In Section~\ref{sect: PA degrees}, we conclude with a number of questions about these  $\PA$ degrees.

\section{Topological Dynamics in $\RCA_0$}

\noindent A standard definition of a topological dynamical system on Cantor space is the following.

\begin{definition}
\label{def: system}
A pair $(C,F)$ is a \textit{topological dynamical system} on $\Cs$ if $C$ is a non-empty closed  subset of $\Cs$, and $F$ is a continuous transformation of $\Cs$ such that for all $X \in C$, $F(X) \in C$. 
\end{definition}

Note that this definition requires $F$ to be defined on all elements of $\Cs$ rather than just those elements of $C$. This is not a limitation because of the Tietze extension theorem which can be proved in $\RCA_0$ \cite[Theorem II.7.5]{Simpson_2009}.
  Sometimes $F\!\until{C}$ is required to be a homeomorphism of $C$ but we will not consider that possibility here.
From now on, we will often refer to a topological dynamical system as simply a system.

We  need to consider how we encode a system inside a model of 2\textsuperscript{nd} order arithmetic. We need to be careful with our choice of encoding to ensure that it behaves sensibly for models of $\RCA_0$. We will follow the standard approach of  encoding a closed set as the set of paths through a tree. The approach we take for encoding a continuous transformation is also standard. 

\begin{definition}
\label{def:endo}
A function $f:\strs \rightarrow \strs$ \emph{encodes a continuous transformation of $\Cs$} if 
\begin{enumerate}
\item $f$ is total.
\item $f$ is order preserving.
\item $(\forall l)(\exists m)(\forall \sigma \in \{0,1\}^m)|f(\sigma)|\ge l$.
\end{enumerate}
\end{definition}
Given such an $f$, we will  denote by $F$ the transformation of Cantor space encoded.  In particular,  if $X \in \Cs$, we will denote by $F(X)$,
\[ \lim_{l \in \omega} f(X\until{l}).\]
It could be objected that the third condition should really be stated as for all $X$ and $l$, there exists an $m$ such that $|f(X\until{m})| >l$, and the equivalence of this statement with the third condition is a consequence of the compactness of Cantor space. However, because we are working over $\RCA_0$  we will use the stronger definition above. 
Let $\lambda$ denote the empty string. For convenience, we will assume that any such $f$ has the additional property that for any string $\sigma \ne \lambda$, $|f(\sigma)| < |\sigma|$. This does not result in any loss of generality because given any $f$ encoding a continuous transformation of $\Cs$, we can uniformly find a $\hat{f}$ with this property that encodes the same continuous transformation of $\Cs$ as $f$.

\begin{definition}[$\RCA_0$]
\label{def: system2}
A pair $(C, f)$ encodes a system  if 
\begin{enumerate}[(i)]
\item The set $C$ is a tree.
\item The set $f$ encodes a continuous transformation of $\Cs$.
\item $[C] \ne \emptyset$.
\item For all $\sigma \in C$ we have that $f(\sigma) \in C$.
\end{enumerate}
\end{definition}

The final condition ensures that if  $(C, f)$ encodes a system inside a model $\M$ then if we extend $\M$ by adding additional reals, then $(C, f)$ still encodes a system inside the extended model. This rules out pathological cases. For example, $f$ could map paths in $[C]$ that are not in the model outside of $[C]$. Then if we extend our model by adding such a path $(C,f)$ would no longer be a system.

It is impossible to discuss recurrence points and almost periodic points without discussing orbits. 
Given a system $(C,f)$ and $X \in C$  we need to show that even with the limited induction available in $\RCA_0$, the orbit of $X$ is well-defined.

Given a function $f$ encoding a continuous transformation of $\Cs$, define  
$f:\strs \times \omega \rightarrow \strs$ by 
\[f(\sigma, k) =\begin{cases}
\sigma & k=0 \\
f(f(\sigma, k-1)) & k>0.
\end{cases}
\]
When convenient, we will write $f^k(\sigma)$ for $f(\sigma, k)$. Observe that $f^1 =f$. We define $F^k(X) = \lim_{l\in \omega} f^k(X\until{l})$. The following lemma shows that given Definitions \ref{def:endo} and \ref{def: system2}, we can talk sensibly about orbits of a point in $\RCA_0$.

\begin{lemma}[$\RCA_0$]
\label{lemma:basic}
If $(C,f)$ is a system then for all $k$ 
\begin{enumerate}
 \item \label{i:system} $f^k$ encodes a continuous transformation of Cantor space. 
\item \label{iii:system} $(C, f^k)$ is a system. 
\item \label{ii:system} For all $X \in \Cs$, $F(F^k(X))=F^{k+1}(X)$. 
 \end{enumerate}
\end{lemma}
\begin{proof}
\eqref{i:system}. The function $f^k$ is total because it has been defined by primitive recursion from $f$. It is order preserving by $\Pi_1$ induction on $k$ for the formula 
\[(\forall k)(\forall \sigma, \tau, \rho, \pi)((\sigma \preceq \tau \wedge \rho = f^k(\sigma)  \wedge \pi = f^k(\tau) )\rightarrow \rho \preceq \pi).\]
The third condition is established by fixing $l$, then inducting on $k$ for the formula
\[ (\forall k)(\exists m)(\forall \sigma \in \{0,1\}^m)(\forall \tau \in \{0,1\}^{<l})(f^k(\sigma) \ne \tau).\] 
This formula holds trivially for the case $k=0$ (take $m >l$). For the case $k=n+1$, let $m_n$ witness the above formula for $k=n$. 
Now let $m$ be such that for all $\sigma \in \{0,1\}^m$, $|f(\sigma)| \ge m_n$. The existence of $m$ comes from our assumption on $f$.  Note that if $\sigma \in \{0,1\}^m$, then making use of the associative law (provable in $\RCA_0$) we have
\[f^{k+1}(\sigma)  = f^k(f(\sigma)) = f^k(\rho)\]
where $|\rho| \ge m_n$ and so $|f^k(\rho)| \ge l$ and in particular $f^k(\rho)$ cannot equal any string of length strictly less that $l$.

\eqref{iii:system}. The first three conditions for $(C, f^k)$ to be are system are met trivially. Fix $\sigma \in C$. We will show that for all $k$, $f^k(\sigma) \in C$.
The set $\{n : f^n(\sigma) \not \in C\}$ is computable. Hence, if it is non-empty, it has a least element $k$. As $\sigma \in C$, $k$ cannot be $0$. Let $\tau = f^{k-1}(\sigma)$ so $\tau \in C$. But $f^k(\sigma) = f(f^{k-1}(\sigma)) = f(\tau) \in C$, a contradiction. 

\eqref{ii:system}. Fix a $k$ and $X \in C$. For any $l$, there exists some $m_l$ such that 
$F^k(X)\until{l} \preceq f^k(X\until{m_l})$. Hence
\begin{multline*}
F(F^k(X)) = \lim_l f(F^k(X)\until{l}) \preceq\lim_l f(f^k(X\until{m_l})) \\   = \lim_l f^{k+1}(X\until{m_l})
=F^{k+1}(X). \qedhere
\end{multline*}
\end{proof}

The orbit of $X$ under $F$ is the sequence $\seq{F^k(X) \colon k \in \omega}$. Note that this is uniform and hence 
$\oplus_{k \in \omega} F^{k}(X)$ exists by recursive comprehension in any model of $\RCA_0$ that includes $X$ and $f$. 

%
%
%

\section{Recurrent Points}
\label{sect: RP}
\noindent We call $X$ a \emph{recurrent point} of a  topological dynamical system $(C, f)$, if 
$X \in [C]$ and 
\[(\forall n, c)(\exists k)(F^{n+k}(X) \succeq X\until{c}).\]

We call $X$ an \emph{almost periodic point} of a  topological dynamical system  $(C, f)$, if 
$X \in [C]$ and 
\[(\forall c)(\exists b)(\forall n)(\exists k <b)(F^{n+k}(X) \succeq X\until{c}).\]

 This leads to two principles. First $\msf{RP}$ is the principle that every topological dynamical system on 
 $\Cs$ contains a recurrent point. The second $\msf{AP}$ is the principle that every topological dynamical system on $\Cs$ contains an almost periodic point.  Over $\msf{RCA_0}$ we have the obvious implication that $\msf{AP}$ implies $\msf{RP}$ because an almost periodic point is a recurrent point. 

The following theorem would be trivial if condition (iii)  in Definition~\ref{def: system2}, was replaced by requiring the tree $C$ to be infinite. Given such a $C$, we could simply take $(C,f)$ be our system where $f$ is the identity map. Any recurrent point of this system would have to be an element of $[C]$ hence proving $\WKL$. 


\begin{theorem}
Over $\RCAO$, $\msf{RP}$ implies $\WKL$.
\end{theorem}
\begin{proof}
Let $T$ be an infinite computable tree on $2^{<\omega}$. We will regard $T$ as computable tree in $3^{<\omega}$ (i.e.\ a computable subtree of $3^{<\omega}$ such that no node of $T$ contains a $2$). 
We will define a  system on $3^\omega$ such that any recurrent point of the system is a path on $T$. As  $3^\omega$ is computably homeomorphic to $2^\omega$, this is sufficient to prove the theorem. The idea behind the following definition of $f$ is that if $X \in 3^\omega$ is a path on $T$, then $F(X)=X$. If $X$ is not a path on $T$, then the orbit of $X$ moves in increasing lexicographical order searching for a path on $T$, looping around if it extends $2$.  The extra branching of  $3^\omega$ allows us to move the orbit of $F(X)$ if $X$ is not a path on~$T$.

Define the following function $f:3^{<\omega} \rightarrow 3^{<\omega}$. First $f(\lambda) = \lambda$. Second if $|\sigma| >0$, let $n = |\sigma| -1$. If $\sigma \in T$, let $f(\sigma) = \sigma \until{n}$. If $\sigma \not \in T$, then let $\pi$ be the shortest initial segment of $\sigma$ such that $\pi \not \in T$. Because $T$ is a subtree of $2^{<\omega}$, if $\pi$ contains a $2$, then $\pi$ must end with $2$. Define
\[f(\sigma) = 
\begin{cases}
\rho 1 0^\omega \until{n} & \pi = \rho 0 \vee \pi = \rho 02 \\
\rho 2 0^\omega \until{n} & \pi = \rho 1 \vee \pi = \rho 12 \\
0^n & \pi =2.
\end{cases}\]

It is not difficult to verify that $(3^\omega, f)$ is a system. Let $\le_{lex}$ be the lexicographical ordering on finite strings.  (Recall that under this ordering $\sigma \le_{lex} \tau$ if $\sigma \preceq \tau$ or $\sigma(i) < \tau(i)$ for the least $i$ where these strings differ.)

\begin{claim} Let $n >0$. Let $\sigma_0, \sigma_1, \ldots, \sigma_n$ be a finite sequence of strings such that $\sigma_0 \succ \sigma_n$, for all $i <n$, $f(\sigma_i)=\sigma_{i+1}$ and $\sigma_n \not \in T$. Then for some $k < n$, $2 \preceq \sigma_k$.
\end{claim} 
\begin{proof}
Consider $S = \{ i \le n \colon \sigma_i \le_{lex} \sigma_n \wedge \sigma_i \not \in T\}$. The set $S$ is not empty as it contains $n$. As $S$ is computable, it has a least element $l$. 
Now $l \ne 0$ as $\sigma_0$ is a strict extension of $\sigma_n$. Let $k=l+1$. First $\sigma_k \not \in T$ as otherwise $f(\sigma_k) = \sigma_l \in T$. By minimality of $l$ we have that $\sigma_k \not \le_{lex} \sigma_n$ and in particular $\sigma_k \not \le_{lex} \sigma_l$. Now because $\sigma_k \not \in T$, the definition of $f$ implies  that $\sigma_k \succeq 2$.
%
%
%
\end{proof}

\begin{claim}If $|\tau|> |\sigma|$, $\sigma <_{lex} \tau$, and  $\tau \in T$ then 
$f(\sigma)<_{lex} \tau$.
\end{claim}
\begin{proof}
If $\sigma \prec \tau$, then $\sigma \in T$ and so $f(\sigma)$ is an initial segment of $\sigma$ and  the result holds.
Otherwise let $\xi$ be the least common initial segment of $\sigma$ and $\tau$. So $\xi0 \preceq \sigma$. Because $\xi$ is on the tree, either $f(\sigma)$ extends $\xi0$ or $f(\sigma) = \xi10^j$ for some $j$. However as $|\tau|> |\sigma| \ge |f(\sigma)|$ this implies that in either case $f(\sigma) <_{lex} \tau$.
\end{proof}

Let $R$ be a recurrent point for this system. Assume  $R  \not \in [T]$. Take $\sigma  \prec R$ such that $\sigma \not \in T$. As $R$ is a recurrent point, there exists a sequence 
$\sigma_0, \ldots, \sigma_n$ such that $\sigma_n = \sigma \prec \sigma_0 \prec R$ and 
$f(\sigma_i) =\sigma_{i+1}$. By the  1\textsuperscript{st} claim for some $k<n$, $2 \preceq \sigma_k$.

Let $\tau \in T$ such that for all $i \le n$, $|\tau| > |\sigma_i|$. Now $\sigma_{k+1}$ is a string of all $0$'s,  and  $|\sigma_{k+1}| <|\tau|$. Hence  $\sigma_{k+1} <_{lex} \tau$.  It follows that $\sigma_n <_{lex} \tau$ by inducting over the 2\textsuperscript{nd} claim. Now as $\sigma_0 \succeq \sigma_n$ and $\sigma_n \not \preceq \tau$ because $\sigma_n \not \in T$, this implies that 
$\sigma_0 <_{lex} \tau$.

But this is impossible. If $\sigma_0 <_{lex} \tau$ then again by inducting over the 2\textsuperscript{nd} claim, for all $i$, $\sigma_i <_{lex} \tau$ and so $\sigma_i \not \succeq 2$. This contracts  the fact that $\sigma_k \succeq 2$.  
Hence $R \in [T]$.
\end{proof}

\begin{theorem}Over $\msf{RCA_0}$, $\msf{WKL}$ implies $\msf{RP}$. 
\end{theorem}
\begin{proof}
Let $(C, f)$ be a system. 
We can define the set of recurrent points of this system, $\mathcal{R}$, as follows.
\[\mathcal{R} = \{ X \in [C] \colon (\forall c)( \exists n, l > c )(f^n(X\until{l}) \succeq X\until{c})\}.\] 
This shows that $\mathcal{R}$ is  $\Pi^0_2$ in  $C\oplus f$. 
In order to prove that $\msf{WKL}$ implies $\msf{RP}$, it is sufficient to  show  there is a non-empty $\Pi^0_1(C\oplus f)$ class contained in $\mathcal{R}$.
We will construct a computable sequence of finite sets of strings
$\seq{ U_i \colon i\in \omega}$ and let 
$\bigcap_i [U_i] \cap [C]$ be our $\Pi^0_1(C\oplus f)$ class. We will ensure that if $X \in [U_i] \cap C$, then for some  $n, l>i$, $f^n(X\until{l}) \succeq X\until{i}$, hence if $X \in  \bigcap_i [U_i] \cap [C]$, then $X\in \mathcal{R}$ and so $X$ is a recurrent point of $(C,f)$. 

The difficulty with defining $U_i$, is that it is possible that $[U_i] \cap [C]$ might be empty.  To avoid this occuring, we will ensure that for all $i$, there is an $s_i$ such that, $\bigcup_{n<s_i} F^{-n}([U_i]) \supseteq [C]$. This means that no $[U_i]$ can be removed entirely from $[C]$ because otherwise  $[C]$ would  either be empty or, for some $X\in [C]$, there would be some $n$ such that $F^{n}(X) \not \in [C]$. In either case $(C,f)$ would not be a system. 

Let $U_0 =V_0 =\{\lambda\}$ and $s_i=0$. We will assume that we are given $U_i$ and $V_i$,  both finite sets of strings and $s_i$ a number such that:
\begin{enumerate}
\item $(\forall \tau \in U_i) (\exists n) 
(i \le n \le s_i \wedge f^n(\tau) \succeq \tau \until{i})$.
\item $(\forall \sigma \in V_i)(\exists n \le s_i)
(\exists \tau \in U_i) (f^n(\sigma) \succeq  \tau)$.
\item $V_i$ is an open cover of $C$.
\end{enumerate}
These conditions hold trivially for the case $i=0$. We inductively define $U_{i+1}[s]$ and $V_{i+1}[s]$ as follows. 
\begin{align*}
U_{i+1}[s] = \{ \sigma \in 2^{<\omega}  &\colon ((\exists \tau \in U_i) (\sigma \succ \tau))  \wedge \\
& (\exists n)((i \le n \le s )\wedge(f^n(\sigma) \succeq \sigma \until{i})) \}.\\  
V_{i+1}[s] = \{  \sigma \in 2^{<\omega} &\colon (\exists n \le s)(\exists \tau \in U_{i+1}[s])( f^n(\sigma) \succeq \tau)\}.
\end{align*}
\noindent These definitions imply that: 
\begin{enumerate}
\item $[U_{i+1}[s]] \subseteq [U_i]$.
\item $[U_{i+1}[s]] \subseteq [U_{i+1}[s+1]]$.
\item $[V_{i+1}[s]] \subseteq [V_{i+1}[s+1]]$.
\end{enumerate}

\begin{claim}
 $\bigcup_s [V_{i+1}[s]] \supseteq C$. 
\end{claim}
\begin{proof}
By applying bounding, there is some  $h > \max\{|\tau| \colon \tau \in V_i\}$ such that 
\[(\forall \sigma \in \{0,1\}^h)(\forall m \le s_i)(|f^m(\sigma)| \ge i+1).\] 
 Take $X \in C$. By the pigeon-hole principle there is some $\sigma \in \{0,1\}^h$ and $j, k \in \omega$ such that $F^j(X) \in [\sigma]$
and $F^{j+k}(X) \in [\sigma]$.
We can also ensure that $k \ge i+1$.
 Now applying $\WKL$, we know that $\sigma$ extends some element of $V_i$. From the definition of $V_i$, this means that for some $m \le s_i$,  $f^m(\sigma)$ extends some $\tau \in U_i$. Let $Y= F^{j+m}(X)$ so $Y \succ f^m(\sigma) \succeq \tau$. Further
$F^k(Y) = F^{j+k+m}(X) \succ f^m(\sigma) \succeq \tau$ as well. Finally, 
$Y\until{i+1} = F^k(Y)\until{i+1}$ because $|f^m(\sigma)| \ge i+1$. Take $l$ such that $f^k(Y\until{l}) \succeq f^m(\sigma)$. 
Thus $Y\until{l}  \in U_{i+1}[\max\{k,l\}]$ and $X \in [V_{i+1}[s]]$ where $s >\max\{j+m,k,l\}$ is large enough such that $f^{j+m}(X\until s) \succeq Y\until{l}$.
\end{proof}
 
Hence by compactness, there is some  least $s_{i+1}$ such any string of length $s_{i+1}$ in $C$ extends some element of  $V_{i+1}[s_{i+1}]$. We define $U_{i+1} =  U_{i+1}[s_{i+1}]$ and $V_{i+1} =  V_{i+1}[s_{i+1}]$. Note that $[U_{i+1}]$ is a closed set in Cantor Space. 
Fix $i$. Take any $X \in [C]$. We know that $X \in [V_i]$ and hence there exists some $n \le s_i$ such that $F^n(X) \in [U_i]$. By Lemma~\ref{lemma:basic}, $F^n(X) \in [C]$. Hence $[U_i] \cap [C] \ne \emptyset$. 
Thus $\seq{[U_i] \cap [C] \colon i \in \omega}$ is a nested sequence of non-empty closed sets and so by $\WKL$ contains an element $R$, which is a recurrent point of $(C,f)$.
\end{proof}

Note that the proof given uses the fact that the sets $[U_i]$ are clopen. The proof can be extended to certain spaces which do not have a basis of clopen sets, such as the unit interval, by adding the condition that for any open set $E  \in U_{i+1}$, we have that the closure of $E$ is contained in  $[U_i]$. To find our recurrent point, we take a point  $R \in \bigcap_i (\cl[U_i] \cap [C])$. The same argument shows that this set is not empty (instead of adding $Y\until{l}$ to $U_{i+1}$ find an $\epsilon$ such that $\cl[B(Y; \epsilon)] \subset [U_i]$ and add this ball to  $U_{i+1}$). However, $R$ must be an interior point of each set $\cl[U_i]$ and  so as $R\in \bigcap_i [U_i]$, $R$ is a recurrent point.

\section{Minimal Systems}
\label{section min}

\noindent We  now investigate the principle $\msf{AP}$.
The standard proof that every topological dynamical system has an almost periodic point uses the existence of minimal subsystems. We call a system $(C,f)$ \emph{minimal}, if for any system $(D,f)$ such that $[D] \subseteq [C]$, we have that $[D]=[C]$.

Let $(C, f)$ be a system. By Zorn's lemma, $(C, f)$ contains a minimal subsystem $(D,f)$. Now by the following standard lemma, which we can formalize in $\msf{WKL_0}$, every point of $D$ is an almost periodic point of $(D,f)$ and hence an almost periodic point of  $(C, f)$.

\begin{lemma}
$\msf{WKL_0}$ proves that any path in a minimal system is almost periodic.
\end{lemma}
\begin{proof}
Let $(D, f)$ be a minimal system and take any $X \in [D]$. Assume $X$ is not an almost periodic point. If so  there exists some $\tau \prec X$ such that 
\begin{equation}
\label{e: notap}
(\forall b )(\exists n)(\forall k \le b)F^{n+k}(X) \not \succ \tau.
\end{equation} 

Define $E = \{\sigma \in D \colon (\forall n \le |\sigma|)(f^n(\sigma) \not \succeq \tau)\}$.
As $X \not \in [E]$, we have that  $[E] \subsetneq [D]$. Our assumption that $|f(\sigma)| <|\sigma|$, implies that for all $\sigma \in E$, $f(\sigma) \in E$. Now $(E,f)$ cannot be a system because this would contradict the minimality of $(D,f)$. 
This means that $[E] = \emptyset$. Applying $\WKL$, there exists a $b$, such that $E$ contains no string of length $b$.  Hence for all $Z \in [D]$ there is some $k <b$ with $F^k(Z) \succ \tau$. This contradicts \eqref{e: notap} and  hence our assumption that $X$ is not an almost periodic point is incorrect.
\end{proof}

 While we appealed to Zorn's lemma to construct a minimal subsystem, this is not necessary for systems in Cantor space. The reason is Cantor space contains a computable basis of open sets. This allows us to show that $\ACA_0$ implies that any system contains a minimal subsystem. 
 To find a minimal subsystem simply enumerate the basis and ask in order can any element be removed. In particular, let $\{\sigma_i\}_{i \in \omega}$ enumerate the finite strings. Given a system $(C,f)$ let $C_0 =C$. If $(C_i, f)$ has been defined,  let $C_{i+1}$ be equal to $\{\tau \in C_i \colon (\forall n \le |\tau|)(f^n(\tau) \not \succeq \sigma_i)\}$ if the later set is not finite.  Otherwise let $C_{i+1} = C_i$. It is not difficult to verify that $( \bigcap_i C_i, f)$ is a minimal subsystem of $(C,f)$. This gives us the following result.

\begin{proposition}
\label{prop: ACA}
$\ACA_0$ proves that any system contains an almost periodic point. 
\end{proposition}

\begin{theorem}
\label{thm: ACA minimal}
Over $\msf{WKL_0}$, $\ACA$ is equivalent to statement that every system contains a minimal subsystem.
\end{theorem}
\begin{proof}
The argument proceeding Proposition~\ref{prop: ACA} shows that $\ACA_0$ proves that every system contains a minimal subsystem. To show the other direction we will work over $\WKL_0$ as our base system. In order to simplify the exposition of this proof, we will work with $\Pi^0_1$ classes of reals in Cantor space as opposed to trees in $2^{<\omega}$.

First we will show how to encode a single bit of $\zj$ into a system. 
Let $f$ be the left-shift. Fix $n$, we will define a $\Pi^0_1$ class $C$ such that given any minimal subsystem $(D,f)$ of $(C,f)$ the set $[01] \cap D$ is empty if and only if $n \not \in \zj$.
In particular, if $n \not \in \zj$, then 
\[C = \{0^i1^\omega \colon i \in \omega\} \cup \{1^i0^\omega \colon i \in \omega\}.\]
Observe that in this case, the only minimal subsystems of $(C,f)$ will be $(\{0^\omega\}, f)$ and $(\{1^\omega\}, f)$.

Let $S_i = \{F^n((0^i1^i)^\omega) : n \in \omega\}$.  Each $S_i$ is a minimal system with $2\cdot i$ elements. For example, 
\[S_2 =\{ (0011)^\omega, 011(0011)^\omega, 11(0011)^\omega, 1(0011)^\omega\}.\] 
If $n \in \zj$, then we will define $C$ to be equal to $S_i$ for some $i$ compatible with our definition of $C$ at the stage $n$ enters $\zj$. 
Formally, let $t$ be $\infty$ if $n\not \in \zj$ and let $t$ be the  least $s$ such that  that $n \in \zj[s]$ otherwise.
Let $E_s =\{X  \in 2^\omega \colon (\exists l \le s)(0^l1^{s-l} \prec X \vee  1^l0^{s-l} \prec X\}$.
 Define 
\[C = \begin{cases} \bigcap_s E_s & t=\infty \\
S_i & t =i <\infty.
\end{cases}\]
Observe that $S_i \subseteq \bigcap_{s< i} E_s$. Hence $C$ is a $\Pi^0_1$ class.
Let $(D,f)$ be a minimal subsystem of $(C,f)$. To determine  if $n$ is in $\zj$ wait until a stage $s$ such that either $n  \in \zj$ or $[01] \cap D[s] = \emptyset$ (the existence of such an $s$ when $n \not \in \zj$ requires $\WKL$). 

In order  to code all elements of $\zj$, we use the uniformity in the definition above to build a product system as follows. 
For all $n$, let $C_n$ be the set defined by the above construction. 
Let $C = \Pi_n C_n$ (i.e.\ $X \in C$ if and only if for all $n$, $X^{[n]} \in C_n$ where  $X^{[n]}$ denotes the $n$th column of $X$). Let $f$ be the mapping produced by applying the left-shift to each column.  Now  if $(D, f)$ is a minimal subsystem of $(C,f)$  then we have that $n \not \in \zj$ if and only if the set
$\{X \in D | X^{[n]} \succ [01]\}$ is empty. Using $\WKL_0$,  this set is empty if and only if  the associated tree is finite and we have provided a $\Sigma^0_1$ definition of the complement of $\zj$. 
\end{proof}

\section{Separating $\msf{AP}$ from $\msf{WKL}$}
\label{sect: AP WKL sep}

\noindent We have seen that $\ACA_0$ proves $\msf{AP}$. Further $\RCA_0 +\msf{AP}$ proves $\WKL$ because any almost periodic point is a recurrent point.
In this section we will separate  $\msf{AP}$ from $\WKL$. 
We will show that there is a model of $\WKL$ that is not a model $\RCA_0 +\msf{AP}$.  The natural numbers in this model will be the true natural numbers and so  we will work with full induction. We will also regard our closed sets  as  $\Pi^0_1$ classes, as this simplifies the exposition.

The key to the separation is the following technical lemma. Let $(C,f)$ be a system. A point $X\in C$ is called a  \emph{periodic point} of $(C,f)$ if for some $n$, $F^n(X) =X$. Let $Orb(X)$ be the orbit of $X$. Note that if $X$ is a periodic point then $Orb(X)$ is a finite set.
For a finite string $\sigma$ we denote by $\sigma^n$, the string obtained by repeating $\sigma$ $n$ times and we denote by $\sigma^\omega$, the infinite sequence $\bigcup_n \sigma^n$.
\begin{lemma} 
\label{lemma: no closed set}
Let $f$ be the left-shift on Cantor space. Let $P \subseteq \Cs$ be a $\Pi^0_1$ class. There is a $\Pi^0_1$ class $C$, computable uniformly in an index for $P$ such that $(C,f)$ is a system and either:
\begin{enumerate}
\item $C \cap P = \emptyset$; or \label{propc1}
\item There is a non-empty $\Pi^0_1$ class $\widehat{P} \subseteq P$ with the property that no element of $\widehat{P}$ is an almost periodic point of $(C,f)$. \label{propc2}
\end{enumerate}
\end{lemma}
\begin{proof}
The definition of $C$ is simple. Let $\{X_i\}_{i\in \omega}$ be an enumeration of the periodic points in $(2^\omega, f)$. Such an enumeration exists because any periodic point is of the form $\sigma^\omega$ for some finite string $\sigma$. We let $C=2^{\omega}$ unless at some least stage $s$ we have $Orb(X_i) \cap P = \emptyset$ for some $i <s$. If so we let $C = Orb(X_i)$ for the least $i$ for which this holds at stage $s$. The definition of $C$ is uniform because we can refine $C$ to $Orb(X_i)$ at any point.

If $C= Orb(X_i)$ for some periodic point $X_i$, then $C \cap P = \emptyset$ and condition \eqref{propc1} is meet. Hence we will consider the case that $C =2^\omega$. 
%
%
%
If there is some computable point $X \in P$ such that $X$ is not almost periodic, then condition \eqref{propc2} holds by defining $\widehat{P} = \{X\}$. Hence we will assume that any computable point in $P$ is almost periodic. 

We  inductively define a sequence of finite strings $\sigma_1, \sigma_2, \ldots.$  The strings will have the following properties.
If $i<j$ then $\sigma_i \preceq \sigma_j$. For all $i$, $1^i$ is a substring of $\sigma_{i+1}$ but $1^{i+1}$ is not. The string $\sigma_1= 0^n$ for some $n >0$.

\begin{enumerate}
\item The sequence $10^\omega$ is computable and not almost periodic. Hence $10^\omega \not \in P$ and so  there exists some $n_1 >0$ such that $[1 0^{n_1}] \cap P= \emptyset$. Let $\sigma_1 = 0^{n_1}$.
\item The sequence $11(\sigma_1 1)^\omega$ is computable and not almost periodic (the subsequence $11$ only occurs once). Hence $11(\sigma_1 1)^\omega\not \in P$, and so there exists some $n_2>0$ such that $[11(\sigma_1 1)^{n_2} ] \cap P= \emptyset$.
Let $\sigma_2 = (\sigma_1 1)^{n_2}$.
 \item Similarly  $111(\sigma_2 \sigma_1 11)^\omega \not \in P$, and so there exists some $n_3$ such that $[111(\sigma_2 \sigma_1 11)^{n_3} ] \cap P= \emptyset$.
Let $\sigma_3 = (\sigma_2\sigma_1 11)^{n_3}$.
\item In general we define $\sigma_{i+1} = (\sigma_i\sigma_{i-1}\ldots \sigma_1 1^i)^{n_{i}}$ such that
\[[1^{i+1}(\sigma_i\sigma_{i-1}\ldots \sigma_1 1^i)^{n_{i}}] \cap P= \emptyset.\]
\end{enumerate}
%
%
%
%
Consider the periodic systems generated by 
$(\sigma_i)^\omega$. Because $C = 2^\omega$,  for all $i$, there is some $X_i \in Orb((\sigma_i)^\omega) \cap P$.

\begin{claim}For all $i$, $X_i(0)=0$.  
\end{claim}
\begin{proof}
Take any $X_i$.
 Let $k \in \omega$ be the largest number such that $1^k$ is an initial segment of $X_i$. First $k < i$ because any substring of $1$'s in $(\sigma_i)^\omega$ has length less than $i$. Further, by construction if $1^k0$ forms an initial sequence of $X_i$ then $1^k\sigma_k$ forms an initial sequence of $X_i$, but $\sigma_k$ was choosen so that $[1^k\sigma_k] \cap P= \emptyset$. Note here we are using the fact that if $i<j$ then $\sigma_i \preceq \sigma_j$. Hence $X_i(0) =0$.
\end{proof}
\begin{claim} Fix $k$. Let $c_k =2k + \sum_{s=1}^k|\sigma_s|$. Then for all $i > k$, $X_i\until{c_k}$  contains  $1^k$ as a substring.
\end{claim}
\begin{proof}
 Let $\tau = \sigma_k \ldots \sigma_1$. We will show by induction that for $i > k$, $\sigma_i$ is a string of the form  $\tau1^{n_1}\tau1^{n_2}\tau1^{n_3} \ldots \tau1^{n_l}$ where each $n_j \ge k$ for $j \in \{1, \ldots, l\}$. 
First $\sigma_{k+1} = (\sigma_{k}\sigma_{k-1}\ldots \sigma_1 1^k)^{n_{k}}
= (\tau1^k)^{n_{k}}$ and  is clearly of this form. Fix $i \ge k+1$ and assume this holds for all $j \in \{k+1, k+2, \ldots, i\}$. Then
\[\sigma_{i+1} = (\sigma_i\sigma_{i-1}\ldots \sigma_1 1^i)^{n_{i}}
=(\sigma_i\sigma_{i-1}\ldots \sigma_{k+1}\tau 1^i)^{n_{i}}\]
and so has the desired property by induction. As $X_i$ is a left-shift of $(\sigma_i)^\omega$ and $c_k = |\tau| +2k$, the claim holds.
\end{proof}

Let $X$ be an accumulation point of $\{X_i \colon i \in \omega\}$. Hence $X$ is an element of $P$ as $P$ is closed.
The sequence $X$ has the property that 
$X(0)=0$ and for all $k$, the initial segment $X\until{c_k}$ contains a subsequence of $1^k$. Observe that the sequence $\{c_i\}$ is computable. Now define $\widehat{P} \subseteq P$ to be the following $\Pi^0_1$ class
\[\{ X \in P \colon  X(0)=0 \wedge (\forall k) (1^k\mbox{ is a substring of }X\until{c_k})\}.\]
If all the assumptions are meet until this point, $\widehat{P}$ is non-empty and  no element of $\widehat{P}$ is an almost periodic point. Hence condition \eqref{propc2} is met.
\end{proof}

In the proof of the following theorem we will make use of the fact that if $P\subseteq \Cs$ is a $\Pi^0_1$ class and $f:\Cs \rightarrow \Cs$ is a total computable function, then  
both $f(P)$ and $f^{-1}(P)$ are $\Pi^0_1$ classes.

\begin{theorem}
\label{thm: WKL does not imply AP}
$\WKL_0$ does not prove $\msf{AP}$.
\end{theorem}
\begin{proof}
Let $f$ be the left-shift. Let $\{Q_i\}_{i\in \omega}$ be a enumeration of all $\Pi^0_1$ classes. It follows from the uniformity of Lemma~\ref{lemma: no closed set}, that we can build a system
\[ (C,g) = \prod_{e\in \omega}(C_e, f)\]
such that if $Q$ is the $e$\textsuperscript{th}  $\Pi^0_1$ class  then either
 \begin{enumerate}
 \item $\pi_e(Q) \cap C_e = \emptyset$ or
 \item There is a non-empty $\Pi^0_1$ class $\widehat{Q} \subseteq Q$ such that no element of $\pi_e(\widehat{Q})$ is almost periodic,
 \end{enumerate}
   where $\pi_e$ is the projection on the $e$\textsuperscript{th} coordinate (see Theorem~\ref{thm: ACA minimal} for an example of how to encode such a product system). While Lemma~\ref{lemma: no closed set} guarantees the existence of a non-empty  $\Pi^0_1$ subset of $\pi_e({Q})$, no element of which is almost periodic, this can be pulled-back along $\pi_e$ to obtain $\hat{Q}$.
 We will show that there is a set of $\PA$ degree that does not compute an almost periodic point of $(C,f)$.

\subsection*{Construction} At stage $0$, let $P_0$ be  a non-empty $\Pi^0_1$ class of sets of $\PA$ degree. At stage $s+1$, let 
 $\Phi_s$ be the $s$\textsuperscript{th}  Turing functional. If for some $n$ the set
 $\{ X\in P_s \colon \Phi_s^X(n) \uparrow\}$ is not empty, then let $P_{s+1}$ be this set for the least such $n$. 
 
 Otherwise, we have that  $\Phi_s$ is total on all elements of $P_s$. 
 Let $Q = \Phi_s(P_s)$. Now $Q$ is a $\Pi^0_1$ class because there is a total functional that agrees with $\Phi_s$ on the elements of $P_s$. Let $e$ be an index for $Q$ as a $\Pi^0_1$ class. 
 There are two possible outcomes. First  $\pi_e(Q) \cap C_e = \emptyset$ in which case let $P_{s+1}= P_s$ and note that no element of $P_{s+1}$ computes an element of $C$ via $\Phi_s$ let alone an almost periodic element.
The other possible outcome is that  there is some non-empty $\hat{Q} \subseteq Q$ such that no element of $\pi_e(\hat{Q})$ is almost periodic in $C_e$ (and hence no element of $\hat{Q}$ is almost periodic in $C$). For this outcome let 
$P_{s+1} = \{X \in P_s \colon \Phi_s^X \in \hat{Q}\}$. In this case,  $P_{s+1}$ is a non-empty $\Pi^0_1$ class, no element of which computes an almost periodic point in $C$ via $\Phi_s$.

By compactness there is some  $X \in \bigcap_i P_i$. This set $X$ is of $\PA$ degree and $X$ does not compute an almost periodic point of $(C,g)$.  Now it is standard result that there is a model of $\WKL_0$, such that all sets in this model are Turing below $X$. This model does not contain an almost periodic point for the system $(C,g)$ and shows that $\WKL_0$ does not imply $\msf{AP}$. 
\end{proof}

\section{Separating $\msf{ACA}$ from $\msf{AP}$}
\label{sect: ACA AP}

\noindent In this section, we will show that there exists a model of $\msf{RCA_0} +\msf{AP}$ that is not a model of $\msf{ACA}$. To achieve this, we will prove that every topological system on Cantor space has an almost periodic point that is low relative to the system. Because the main theorem of this section is a separation result,  we could make use  of full induction. However, we will restrict ourselves to $I\Sigma_1$ induction so that we can make use of these results in Section~\ref{sect: conservation}.

The objective is to construct an almost periodic point of a system while forcing the jump. Let $(C,f)$ be a system and let $U$ be a c.e.\ set of strings. If there is a subsystem  $(D,f)$ of $(C,f)$ such that 
$[D] \cap [U]= \emptyset$, then we can replace our original system with $(D,f)$. Any almost periodic point in $(D,f)$ is an almost periodic point of $(C,f)$ and we know that such a point cannot meet $U$. 

If we cannot find such a subsystem, then we will show that for some $b$ for all $X \in [C]$ there exists some 
$k<b$ with $F^k(X) \in [U]$. We will use this fact to build a new system $(D,g)$ such that 
$[D] \subseteq [C] \cap [U]$ and for all $X \in [C]$, $G(X) = F^k(X)$ for some $k <b$. We will show that this gives us a certain recurrence  property that allows us to build an almost period point that meets $U$.

\begin{definition}\
\begin{enumerate}
\item Let $f,g:2^{<\omega} \rightarrow 2^{<\omega}$ encode continuous transformations of $\Cs$.   Call  $g$  a \textit{piece-wise combination of iterates of }$f$ if for some $l, b$ there is a function $j:\{0,1\}^l \rightarrow \{1, \ldots, b\}$ such that for all $\sigma$ with $|\sigma| \ge l$, 
$g(\sigma) = f(\sigma, j(\sigma\until{l}))$.
\item Let $(C,f), (D,g)$ be systems. We say that  
 $(D,g)$ \emph{refines} $(C,f)$, written  $(D,g) \le (C,f)$ if:
 \begin{enumerate}
 \item $D \subseteq C$.
 \item  $g$  is a {piece-wise combination of iterates of }$f$.
 \end{enumerate}
\end{enumerate}
\end{definition}

Clearly if $(D,f)$ is a subsystem of $(C,f)$ then $(D,f) \le (C,f)$. 

\begin{lemma}$(\WKL_0)$
\label{lemma: nbh ap}
Let $(C,f), (D,g)$ be systems such that $(D, g) \le (C,f)$. If $X \in [D]$, then 
$(\exists b)(\forall n)(\exists k \le b) F^{n+k}(X) \in [D]$. 
\end{lemma}
\begin{proof}
Assume this fails for some $X \in [D]$. Let $b$ witness that $(D, g) \le (C,f)$ i.e.\ $\{1, \ldots, b\}$ is range of the function $j$. Consider the set of $n$ such that  
\[\{F^{n+1}(X), F^{n+2}(X), \ldots, F^{n+b}(X)\} \cap [D]= \emptyset.\]
This set is computable in $X$ and non-empty by assumption. Hence it contains a least element $l$.
By minimality and the fact that $F^0(X) = X \in [D]$ we must have that $F^l(X) \in [D]$ hence for some $k \in \{1, \ldots, b\}$, 
$G(F^l(X))=F^k(F^l(X))$ and so  
$F^{l+k}(X) = G(F^l(X))  \in [D]$ contradicting our assumption.
\end{proof}

%

\begin{lemma}$(\WKL_0)$
\label{lemma: transitive}
The refinement relation is transitive.
\end{lemma}
\begin{proof}
Let $(E,h) \le (D,g) \le (C,f)$. Clearly $[E] \subseteq [C]$. Let $j_1:\{0,1\}^{l_1} \rightarrow \{1, \ldots, b_1\}$ and 
$j_1:\{0,1\}^{l_2} \rightarrow \{1, \ldots, b_2\}$ be such that for all $\sigma$, if $|\sigma| \ge \max\{l_1, l_2\}$ then
\begin{enumerate}
\item $g(\sigma) = f^{j_1(\sigma\until{l_1})}(\sigma)$.
\item  $h(\sigma) = g^{j_2(\sigma\until{l_2})}(\sigma)$.
\end{enumerate}
Let $b_3 = b_1 \cdot b_2$. Let $l_3 >l_2$ be sufficiently large such that for all $\sigma \in \{0,1\}^{l_3}$, for all $n < b_2$, $|g^n(\sigma)| > l_1$.

Take any string  $\sigma$ such that $|\sigma| \ge {l_3}$. Let $m = j_2(\sigma \until{l_2})$. Then 
 $h(\sigma) = g^m(\sigma)$. Further
\begin{align*}
g^m(\sigma) =& g\circ g^{m-1}(\sigma) \\
=& f^{j_1(g^{m-1}(\sigma)\until{l_1})} \circ g^{m-1}(\sigma)\\
=&  f^{j_1(g^{m-1}(\sigma)\until{l_1})} \circ f^{j_1(g^{m-2}(\sigma)\until{l_1})} \circ \ldots 
\circ f^{j_1(g^{0}(\sigma)\until{l_1})}(\sigma)
\end{align*}
Hence $h(\sigma)= f^n(\sigma)$, where $n = \sum_{i=0}^{m-1} j_1(g^{i}(\sigma)\until{l_1})$. Because $m \le b_2$,  we have that $n$ only depends on $\sigma \until{l_3}$. Further $n \le b_3$. Hence $h$ is a piece-wise combination of iterates of $f$ and so
$(E, h) \le (C, f)$.
\end{proof}

\begin{lemma} $(\WKL_0)$
\label{lemma: meet or avoid}
Let $(C,f)$ be a system and $U$ a c.e.\ set. There is a system $(D,g)$ refining $(C,f)$ such that 
either:
\begin{enumerate}
\item \label{out} $[D] \cap [U]= \emptyset$; or
\item \label{in} $[D] \subseteq [U]$.
\end{enumerate}
\end{lemma}
\begin{proof}
Define
\[D_0 = \{ \sigma \in C \colon (\forall n \le |\sigma|)(\forall \tau \in U[|\sigma|])
(f(\sigma,n) \not \succeq \tau)\}.\]
To establish that $D_0$ is a tree, let $\sigma$ and $\sigma'$ be any two strings such that $\sigma \preceq \sigma'$. Assume $\sigma \not \in D_0$. If $\sigma \not \in C$ then because $C$ is a tree $\sigma' \not \in D_0$. If $\sigma \in C$ then for some $n \le |\sigma|$ and $\tau \in U[|\sigma|]$, $f(\sigma, n) \succeq \tau$. Hence $f(\sigma', n) \succeq \tau$ and $\tau \in U[|\sigma'|]$. Thus $\sigma' \not \in D_0$.

\begin{claim}
For all $\sigma \in D_0$, $f(\sigma) \in D_0$.
\end{claim}
\begin{proof}[Proof of claim]
If $f(\sigma) \not \in D_0$, then let $\sigma' =f(\sigma)$. There is some $n \le |\sigma'|$ and $\tau \in U[|\sigma'|]$ such that 
$f(\sigma', n) \succeq \tau$. Hence $f(\sigma, n+1) \succeq \tau \in U[|\sigma|]$. We have that $\sigma \not \in D_0$ because $n+1 \le |\sigma|$ (here we use our assumption that
$|f(\sigma)| <|\sigma|)$.
\end{proof}
This claim establishes that if $D_0$ is infinite, then $(D_0, f)$ refines $(C,f)$, and by the definition of $D_0$ with $n=0$, we have that $[D_0] \cap [U]= \emptyset$.

Now consider the case that $D_0$ is finite. Let $s$ be least such that $D_0$ contains no string of length $s$. Define
\[ D_1 = \{ \sigma \in C \colon (|\sigma| <s) \vee (\exists \tau \in U[s]  ( \sigma \succeq \tau))\}.\]
We will show that $D_1$ is infinite.  Take any $X \in [C]$ and let $\sigma = X\until{s}$. As $\sigma \not \in D_0$, there is a $k \le s$ such that $f^k(\sigma) \succ \tau \in U[s]$. Hence for all $n$, $f^k(X\until{n}) \in D_1$.

Let $l$ be such that if $|\sigma| = l$, then $|f(\sigma)| \ge s$.
Define $j:\{0,1\}^l \rightarrow \{1, \ldots, s+1\}$ by
\[j(\sigma) = \begin{cases}
1 & \sigma \not \in C \\
k& \mbox{where }k \ge 1\mbox{ is  least  such that }(\exists \tau \in U[s] )f(\sigma, k) \succeq \tau\mbox{ if }\sigma \in C.
\end{cases}\]
Note that $j(\sigma)$ is well-defined because for $\sigma \in \{0,1\}^l$ we have that $|f(\sigma)| \ge s$ and so $f(\sigma) \not \in D_0$. Hence for some $k \in \{0, \ldots, s\}$ we have that $f(f(\sigma, k))$ extends some element of $U[s]$.
Define a function $g:2^{<\omega} \rightarrow 2^{<\omega}$ by
\[g(\sigma) = \begin{cases}
\lambda &\mbox{if }|\sigma| < l\\
f(\sigma, j(\sigma \until{l})) & \mbox{otherwise.} \end{cases}\]

It follows from the definition of $j$ and $D_1$, that if  $\sigma \in [C]$, then $g(\sigma) \in [D_1]$. Hence $(D_1, g)$ is a system. Clearly, $g$ is a piece-wise combination of iterates of $f$ and hence $(D_1,g) \le (C,f)$. Finally $[D_1] \subseteq U[s]$. 
\end{proof}

The proof given of the proceeding lemma provides some more information that we will make use of in Section~\ref{sect: conservation}. We state this as the following lemma.

\begin{lemma}\label{lemma: inout} $(\WKL_0)$ \
Consider the set 
\[\{ \sigma \in C \colon (\forall n \le |\sigma|)(\forall \tau \in U[|\sigma|])
(f(\sigma,n) \not \succeq \tau)\}.\]
 Case \eqref{out} of Lemma~\ref{lemma: meet or avoid} holds if this set is infinite.
 Case \eqref{in} of Lemma~\ref{lemma: meet or avoid}  holds if this set is finite, and further there is a $(D,g)$ refining $(C,f)$ such that for all $X \in [C]$ there is a $k$ with $F^k(X) \in [D]$.
\end{lemma}

We make use of full induction for the following lemma.
\begin{lemma}
Any system  $(C,f)$ contains an almost periodic point $X$ such that $X' \le_T (C\oplus f)'$.
\end{lemma}
\begin{proof} We define a sequence of systems $\{(C_e, f_e)\}_{e \in \omega}$ such that for all $e$, 
$(C_{e+1}, f_{e+1}) \le (C_e, f_e)$. 
Let $(C_0, f_0) = (C,f)$.
At stage $e+1$, let $U_e = \{ \sigma \in 2^{<\omega} \colon \Phi_e^\sigma(e) \downarrow\}$. 
Let $(C_{e+1}, f_{e+1})$ refine $(C_e, f_e)$ such that either
$[C_{e+1}] \cap [U_e] = \emptyset$ or  $[C_{e+1}] \subseteq [U_e]$.

An  examination of the  proof of Lemma~\ref{lemma: meet or avoid} shows that this sequence can be constructed below $(C \oplus f)'$. In Lemma~\ref{lemma: meet or avoid}, $D_0$, $D_1$ and $g$ are defined uniformly from $C$ and $f$ (the definition of $D_1$ and $g$ depend on $D_0$ being finite). Further $(C\oplus f)'$, can determine whether or not $D_0$ is finite and hence decide how to refine  $(C,f)$.

By compactness, $\bigcap_e [C_e]$ is not empty. In fact $\bigcap_e [C_e]$ contains a unique point $X$, because every finite set occurs as infinitely many c.e.\ sets $U_e$. Now $X' \le_T  (C\oplus f)'$ because whether $\Phi_e^X(e)$ halts can be determined at stage $e$ of the construction. We show that $X$ is an almost periodic point of $C$. Fix $\sigma \prec X$. Now for some $e$, $[\sigma] \supseteq [C_e]$.  Thus  by Lemma~\ref{lemma: nbh ap} there is some bound $b$, such that for all $n$, there is some $k \le b$ such that 
$F^{n+k}(X) \in [C_e] \subseteq [\sigma]$.  Hence $X$ is an almost periodic point of $(C,f)$.
\end{proof}

Using the standard approach, the previous proposition can be used to build an $\omega$-model of $\RCAO$ and $\msf{AP}$ such that every real in the model is low. Hence we obtain the following theorem.

\begin{theorem}
\label{thm: AP does not imply ACA}
There is an $\omega$-model of $\RCA_0$ and $\msf{AP}$  that is not a model of $\ACA$.
\end{theorem}

\section{A Conservation Result}
\label{sect: conservation}
\begin{lemma}[$\WKLO$]
\label{lemma: preservation induction}
Let  $(C,f)$ be a system, $P$ be a real, and $\varphi$ be a $\Delta_1$ formula. There is a system $(D,g)$ refining $(C,f)$ such that either
\begin{enumerate}
\item The set $\{m : (\exists s) \varphi(m,X\until{s}, P)\}$ is empty for  all $X \in [D]$; or
\item There is a $b$ such that for  all $X \in [D]$, $b$ is the least element of $\{m : (\exists s) \varphi(m,X\until{s}, P)\}$.
\end{enumerate}
\end{lemma}
\begin{proof}
For all $n$, define the following sets
\begin{align*}
&O(\le n) = \{\tau \colon (\exists m \le n) \varphi(m,  \tau, P)\},\\
&O(< n) = \{\tau \colon (\exists m < n) \varphi(m, \tau, P)\}, \\
&C(\le n) =\{ \sigma \in C \colon (\forall i \le |\sigma|)(\forall \tau \in O(\le n))(f(\sigma, i)\not \succeq \tau )\},\\
&C(< n) =\{ \sigma \in C \colon (\forall i \le |\sigma|)(\forall \tau \in O(< n))(f(\sigma, i)\not \succeq \tau )\}.
\end{align*}
Because of the uniformity in the above definitions, we have that the following set is c.e.
\[S= \{n \colon C(\le n) \mbox{ is finite}\}.\]
Hence by $I\Sigma_1$ induction $S$ is either empty or contains a least element. First we consider the case that  $S$  is empty. 
Let $U= \{\tau \colon (\exists n)\varphi(n,\tau, P)\}$. 
Let $(D, g)$ be a refinement of $(C,f)$ guaranteed by Lemma~\ref{lemma: meet or avoid}.
If $[D] \cap [U] = \emptyset$, then the system $(D,g)$ has the  required properties.  If $[D] \subseteq [U]$, then by applying compactness there is some bound $b$, 
such that 
\[[D] \subseteq [\{\tau \colon (\exists n \le b)\varphi(n,\tau)\}].\]
We now replace $(C,f)$ by $(D,g)$ and adjust the definitions accordingly. This replacement does not affect the following argument  because of the transitivity of the refinement relation. Note that now $S$ cannot be empty because it must contain $b$. 

Now consider the case when $S$ has a least element $b$. By Lemma~\ref{lemma: inout}, there is a system $(E, h)$ such that 
$(E, h)$ refines $(C,f)$ and $[E] \subseteq [O(\le b)]$. 
Further we have that for all $X \in [C]$ there is a
$k$ such that $F^k(X) \in [E]$. 
Now consider the set 
\[\{\sigma \in E \colon (\forall i \le |\sigma|)(\forall \tau \in O(< b))(g(\sigma, i)\not \succeq \tau )\}.
\]
If this set is empty, then if $X \in [C]$, we have that $F^k(X) \in [E]$ for some $k$ and so for some $j$, $G^j(F^k(X))$ must extend some element of $O(<b)$. This implies that 
$[C] \subseteq [O(< b)]$ as $g$ is a piecewise combination of iterates of $f$. Using compactness, this implies that $C(\le m)$ is empty for some $m$ strictly less than $b$, contradicting the minimality of $b$. Hence the set is not empty and so by Lemma~\ref{lemma: inout} there is a a system $(F,i)$ refining $(E,h)$ (and consequently refining $(C,f)$) such that $[F] \cap  [O(< b)]= \emptyset$. If $X \in [F]$, then $\{m \colon (\exists s) \varphi(m, X\until{s},P)\}$ has least element $b$.  
\end{proof}

\begin{lemma}[$\WKLO$]
\label{lemma: nbh periodicity}
Let $(C,f)$ be a system and $i \in \omega$. There is a subsystem $(D,f)$ of $(C,f)$ and  $b\in \omega$ such that for all $X\in D$
\[(\forall n)(\exists k < b)(F^{n+k}(X) \in [X\until{i}]).\] 
\end{lemma}
\begin{proof}
Take a computable coding of finite sets of finite strings $\{E_k\}_{k\in \omega}$ such that if $E_k \subseteq E_j$, then $k\le j$. Now enumerate a c.e.\ set $W$ by adding $k$ to $W$ if 
\begin{enumerate}
\item The number $k$ is a code for a finite set $E_k \subseteq \{0,1\}^i$.
\item The following tree is finite
\[\{ \sigma \in C  \colon (\forall n \le |\sigma|)(\forall \tau \in E_k)(f(\sigma, n) \not \succeq \tau)\}.\] 
\end{enumerate}
As $C \ne \emptyset$, the code for $\emptyset$ is not an element of $W$. The code for $\{0,1\}^i$ is an element of $W$. Hence by $I\Sigma_1$ induction, there is a maximum element $k$ such that $E_k \subseteq \{0,1\}^i$, and for all $x \le k$ that code subsets of $\{0,1\}^i$, $x \not \in W$.  Let
\[D=\{ \sigma \in C  \colon (\forall n \le |\sigma|)(\forall \tau \in E_k)(f(\sigma, n) \not \succeq \tau)\}.\]
Hence $D$ is infinite and $[D] \subseteq [\{0,1\}^i \setminus E_k]$. 
Now if $\sigma \in  \{0,1\}^i \setminus E_k$, then let $s_\sigma$ be the least number such that 
\[\{ \sigma \in C  \colon (\forall n \le |\sigma|)(\forall \tau \in E_k\cup \{\sigma\})(f(\sigma, n) \not \succeq \tau)\}\]
contains no strings of length $s_\sigma$. Hence for any $X \in [D]$ there is some $n \le s_\sigma$ such that $F^n(X) \in [\sigma]$. In particular, this includes  any element of  $[D] \cap [\sigma]$. The set
$\{(\sigma, s_\sigma) \colon \sigma \in   \{0,1\}^i \setminus E_k\}$ is also c.e.\ and hence by $B\Sigma_1$ induction, there is some $b$ that bounds all elements of this set.
\end{proof}

\begin{definition}
Let $\M$ and $\hat{\M}$ be models of 2\textsuperscript{nd} order arithmetic. Call  \textit{$\M$  an $\omega$-submodel of $\hat{\M}$} if $\M$ is a submodel of $\hat{\M}$ and they share the same first order part.
\end{definition}

\begin{theorem}[Harrington -- unpublished see \cite{Simpson_2009}] Let $\M$ be a countable model of $\RCA_0$. Then there exists a countable model $\hat{\M}$ of $\WKL_0$ such that 
$\M$ is an $\omega$-submodel of $\hat{\M}$.
\end{theorem}

\begin{lemma}
\label{lemma: adding ap point}
Let $\M$ be a countable model of $\WKL_0$ and let $(C,f)$ be a system in $\M$. 
Then there exists a model $\hat{\M}$ of $\WKL_0$ such that $\M$ is an $\omega$-submodel of $\hat{\M}$ and $\hat{\M}$ contains an almost periodic point for the system $(C,f)$.
\end{lemma}
\begin{proof}
By Harrington's theorem it is only necessary to find a model $\hat{\M}$ of $\RCA_0$ such that $\M$ is an $\omega$-submodel of $\hat{\M}$ and $\hat{\M}$ contains an almost periodic point for the system $(C,f)$.
Let $\setN$ be the natural numbers inside $\M$ and let $\setR$ be the reals inside $\M$. From outside of the model, 
let  
$g: \omega \rightarrow \setN$ and $h:\omega \rightarrow \setR$ be bijections.
Define $(C_0, f_0) =(C, f) $. Now inductively define $(C_{e+1}, f_{e+1})$ as follows.\begin{enumerate}
\item If $e =2\cdot \la n, m\ra +1$,  then let $(C_{e+1}, f_{e+1})$ refine $(C_{e}, f_{e})$ as per Lemma~\ref{lemma: preservation induction} with $\varphi$ the $n$\textsuperscript{th} $\Delta_1$ formula and $P=h(m)$.
\item If $e=2\cdot n+2$ then let $(C_{e+1}, f_{i+1})$ refine $(C_{e}, f_{e})$ as per Lemma~\ref{lemma: nbh periodicity} with $i=g(n)$.
\end{enumerate}

Take $R \in \bigcap_e [C_e]$. Let $\hat{\M}$ be the model obtained by adding all reals computable in $R\oplus Y$ to $\M$ for any $Y \in \M$. 
We will show that $\hat{\M}$ is a model of $I\Sigma_1$ induction. Let $W$ be the  $n$\textsuperscript{th} c.e.\ set relative to $R$ with parameter $h(m)$. Let $e=2\cdot \la n, m\ra +1$. If the outcome of  
Lemma~\ref{lemma: preservation induction} was that the set 
\[\{m : (\exists s) \varphi(m,X\until{s}, P)\}\]
is empty for all $X \in [C_{e+1}]$, then the set $\{m : (\exists s) \varphi(m,R\until{s}, P)\}$ is also empty because otherwise a path with this property could be found inside $\M$ using $\WKL$. Similarly if for some $b$,  the  set 
\[\{m : (\exists s) \varphi(m,X\until{s}, P)\}\]
has a least element $b$, for all $X \in [C_{e+1}]$, then the set 
$\{m : (\exists s) \varphi(m,R\until{s}, P)\}$ also has least element $b$.
 
We know that $R$ is an almost periodic point of $(C,f)$ because $R \in [C_{e+1}]$ for $e=2\cdot n+2$ establishes that there exists a $b$ such that for all $m$, there is a $k \le b$ such that 
$F^{m+k}(R) \in R\until{g(n)}$. 
%
%
\end{proof}

\begin{theorem}Let $\M$ be a countable model of $\RCA_0$. Then there exists a countable model $\hat{\M}$ of $\RCA_0 + \msf{AP}$ such that 
$\M$ is an $\omega$-submodel of $\hat{\M}$.
\end{theorem}
\begin{proof}
Let $\M$ be a countable model of $\RCA_0$. Let $g:\omega \rightarrow \omega \times \omega$ be a bijection such that for all $n$, if $(i,j) = g(n)$ then $\max\{i,j\} \le n$. (Note that $\omega$ is the real $\omega$ and $g$ exists outside the model.)

Using Harrington's theorem,  let $\M_0$ be a model of $\WKL_0$ such that $\M$ is an $\omega$-submodel of $\M_0$. Applying  Lemma~\ref{lemma: adding ap point}, let $\M_{n+1}$ be a model of $\WKL_0$, such that  $\M_n$ is an $\omega$-submodel of $\M_{n+1}$, and $\M_{n+1}$ contains an almost periodic point for the $i$\textsuperscript{th} system in $\M_j$ where $(i,j) = g(n)$. Let $\M_\omega = \bigcup_i \M_i$. 

Clearly $\M_\omega \models PA^{-} + \msf{AP}$. It also models $I\Sigma_1$ induction because if not there is some function $f:\setN \rightarrow \setN$ in $\M_\omega$ whose range has no least element. But if so, for some $n$, $f \in \M_n$ contradicting the fact that $\M_n \models \WKL_0$.
\end{proof}

We obtain the following corollary by applying the standard argument (see for example \cite[Corollary~IX.2.6]{Simpson_2009}).
\begin{corollary} $\RCA_0 +\msf{AP}$ is conservative over $\RCA_0$ for $\Pi^1_1$ sentences.
\end{corollary}

\section{A Subclass of PA Degrees}
\label{sect: PA degrees}
\noindent Consider the set of reals that  given any computable system $(C,f)$ can compute an almost periodic point for this system. This is an upwards-closed subclass of the $\PA$ degrees. By Theorem~\ref{thm: WKL does not imply AP}, we know that this is a strict subclass of the $\PA$ degrees. 
The following corollary shows that this subclass does not coincide with those $\PA$ degree above $\zj$

\begin{corollary}[Corollary to Theorem~\ref{thm: AP does not imply ACA}]
There is a set $X$ of $\PA$ degree such that $X \not \ge_T \zj$ and $X$ computes an almost periodic point for every computable system.
\end{corollary}
\begin{proof}
Let $\{Z_i\}_{i\in \omega}$ be a listing of the  ideal used to separate $\ACA$ from $\msf{AP}$ over $\RCA_0$. Observe that no finite join of this sequence computes~$\zj$.

Construct $X$ by at stage $e$ defining sufficient columns of $X$ to force that $\Phi_e^X \ne \zj$, and then append $Z_e$ to an empty column of $X$.  The set $X$ bounds all elements of the ideal so $X$ is of $\PA$ degree.
\end{proof}

The $\PA$ degrees have been extensively studied. However, this subclass does not appear to have been encountered before and it merits further investigation. 

\begin{question}
\label{Q1}
Are there any other characterizations of this subclass?
\end{question}

A useful  answer to Question~\ref{Q1} would give some indication as to how this subclass is dispersed in the Turing degrees. As there are computably dominated sets of $\PA$ degree, it is natural to ask the following question.

\begin{question}
Does this subclass have a computably dominated element?
\end{question}

For this subclass to have a computably dominated element, it is necessary that the following question has a positive answer. 

\begin{question}
Does every computable system have an almost periodic point that is computably dominated?
\end{question}

\bibliographystyle{plain}
\bibliography{../Bibliography/Research}

\end{document}